\documentclass{amsart}
\usepackage{amsmath,amssymb,graphicx,subfig}
\newcommand{\norm}[1]{\Vert #1 \Vert}
\newtheorem{theorem}{Theorem}

\newcommand{\Li}{\text{Li}_2}
\begin{document}
\title{Estimates for the norms of products of sines and cosines}
\author{Jordan Bell}
\email{jordan.bell@gmail.com}
\address{Department of Mathematics, University of Toronto, Toronto, Ontario, Canada}
\date{\today}
\begin{abstract}
In this paper we prove asymptotic formulas for the $L^p$ norms of $P_n(\theta)=\prod_{k=1}^n (1-e^{ik\theta})$ and $Q_n(\theta)=\prod_{k=1}^n (1+e^{ik\theta})$. These products can be expressed using $\prod_{k=1}^n \sin\Big(\frac{k\theta}{2}\Big)$ and $\prod_{k=1}^n \cos\Big(\frac{k\theta}{2}\Big)$ respectively. We prove an estimate for $P_n$ at a point near where its maximum occurs. Finally, we give an asymptotic formula for the maximum of the Fourier coefficients of $Q_n$.
\end{abstract}
\maketitle

\section{Introduction}
Euler's pentagonal number theorem is the expansion
\[
\prod_{k=1}^\infty (1-z^k) = \sum_{k=-\infty}^\infty (-1)^k z^{k(3k-1)/2},
\]
for $|z|<1$. Euler's discovery and proof of it are told in
detail in \cite{MR2651525}. The coefficients in the power series expansion of $\prod_{k=1}^\infty (1-z^k)$ have a combinatorial interpretation that can be used 
to prove the pentagonal number theorem  \cite[pp. 286--287, \S 19.11]{MR2445243}.
One can see that
\[
\prod_{k=1}^\infty (1+z^k)=\sum_{k=0}^\infty q(k) z^k,
\]
where $q(k)$ is the number of ways to write $k$ as a sum of distinct positive integers.

In this paper we are concerned with the behavior on the unit circle of the
partial products of the above infinite products. (The distribution of the zeros of the partial sums of the above infinite series is studied
in \cite{MR1907123}.) Let $\mathbb{T}=\mathbb{R}/2\pi\mathbb{Z}$. We
define $P_n:\mathbb{T} \to \mathbb{C}$ by
\[
P_n(\theta)=\prod_{k=1}^n (1-e^{ik\theta}),
\]
and we define $Q_n:\mathbb{T} \to \mathbb{C}$ by
\[
Q_n(\theta)=\prod_{k=1}^n (1+e^{ik\theta}).
\]

One can check that
\begin{equation}
P_n(\theta)=(-2i)^n e^{\frac{iN\theta}{2}} \prod_{k=1}^n \sin\Big(\frac{k\theta}{2} \Big), \qquad
N=\frac{n(n+1)}{2},
\label{sineproduct}
\end{equation}
and that
\begin{equation}
Q_n(\theta)= 2^n e^{\frac{iN\theta}{2}} \prod_{k=1}^n \cos\Big(\frac{k\theta}{2} \Big),
\qquad
N=\frac{n(n+1)}{2}.
\label{cosineproduct}
\end{equation}

For $f:\mathbb{T} \to \mathbb{C}$, we define the Fourier coefficients of $f$ by
\[
\hat{f}(n)=\frac{1}{2\pi}\int_0^{2\pi} f(\theta)e^{-in\theta}d\theta.
\]
For $1 \leq p < \infty$, we define the $L^p$ norm of $f$ by
\[
\norm{f}_p=\Big(\frac{1}{2\pi} \int_0^{2\pi} |f(\theta)|^p d\theta\Big)^{1/p},
\]
and we define the $\ell^p$ norm of $\hat{f}$ by
\[
\norm{\hat{f}}_p=\Big(\sum_{k=-\infty}^\infty |\hat{f}(k)|^p \Big)^{1/p}.
\]

We deal with $P_n$ in \S \ref{Pnsection}, and we deal with $Q_n$ in \S \ref{Qnsection}. We give combinatorial interpretations of their Fourier coefficients, prove asymptotic formulas for their $L^p$
norms,
present some other approaches for bounding their norms,
and give an asymptotic formula for the $\ell^\infty$ norm of the Fourier coefficients of $Q_n$.
We also prove an estimate
for $P_n$ at a point near where its maximum occurs. In \S \ref{conclusion} we discuss what remains to be shown about these products.

\section{$P_n$}
\label{Pnsection}
The Fourier coefficients of $P_n$ have a combinatorial interpretation.
One can see that
\[
\widehat{P_n}(k)=e_{n,k}-o_{n,k},
\]
where $e_{n,k}$ is the number of ways in which $k$ can be written as a sum
of an even number of positive integers that are distinct and each $\leq n$
and $o_{n,k}$ is the number of ways in which $k$ can be written as a sum
of an odd number of positive integers that are distinct and each $\leq n$. For example,
one can check that
$6+5+2+1,6+4+3+1,5+4+3+2$ are the only ways to write 14 as a sum of an even number of
positive integers that are distinct and each $\leq 6$, so $e_{6,14}=3$, and that
$6+5+3$ is the only way to write 14 as a sum of an odd number of positive integers that are distinct and
each $\leq 6$, so $e_{6,14}=1$. Thus $\widehat{P_6}(14)=2$.

We see from \eqref{sineproduct} that $|P_n(\theta)|=\prod_{k=1}^n 2|\sin\Big(\frac{k\theta}{2} \Big)|$. In Figure \ref{sinefigure} we plot $\prod_{k=1}^{10} 2|\sin (k\theta)|$ for $0 \leq \theta \leq \frac{\pi}{2}$.

\begin{figure}
\includegraphics{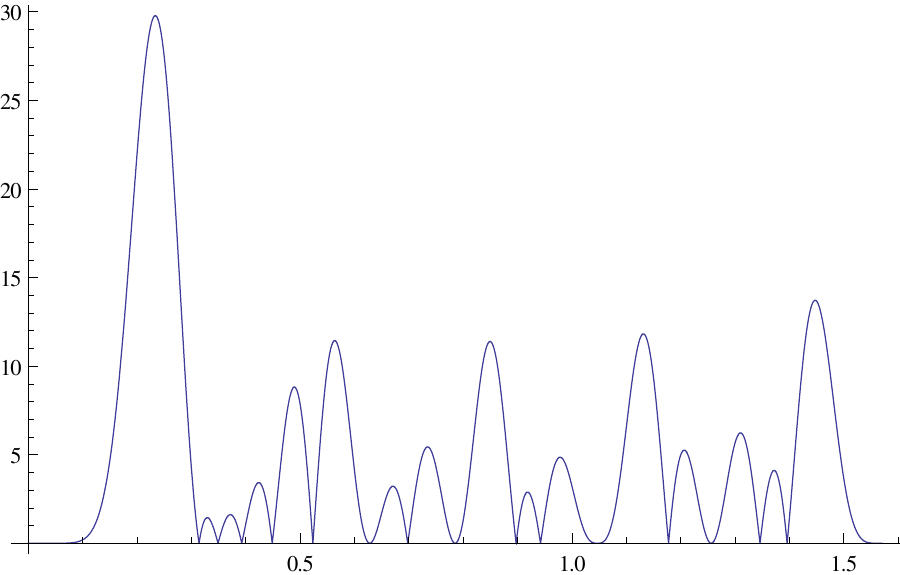}
\caption{$\prod_{k=1}^{10} 2|\sin (k\theta)|$ for $0 \leq \theta \leq \frac{\pi}{2}$}
\label{sinefigure}
\end{figure}

Of course, $P_n(0)=0$. Aside from $\theta=0$ we can explicitly evaluate $P_n(\theta)$ for certain
other $\theta$. For any $h$ such that $\gcd(n+1,h)=1$, we have $z^{n+1}-1=\prod_{k=1}^{n+1} (z-e^{\frac{2\pi ihk}{n+1}})$.
Since $z^{n+1}-1=(z-1)(z^n+\cdots+z+1)$, we get
$z^n+\cdots+z+1=\prod_{k=1}^n (z-e^{\frac{2\pi ihk}{n+1}})$ and setting $z=1$ gives
\[
P_n\Big(\frac{2\pi h}{n+1}\Big)=n+1
\]
for each $h$ such that $\gcd(n+1,h)=1$. In particular this gives us
$\norm{P_n}_{\infty} \geq n+1$.

Wright \cite{MR0163891}, using work of Sudler \cite{MR0163890}, proves the following theorem, which gives an asymptotic formula
for $\norm{\widehat{P}_n}_\infty$.

\begin{theorem}[Wright]
We have
\[
\norm{\widehat{P_n}}_\infty \sim \frac{B e^{Kn}}{n},
\]
where $B$ and $K$ are defined as follows (see Sudler \cite{MR0163890}):
\[
K=\log 2 + \max_{0<w<1} \Big( w^{-1}\int_0^w \log \sin (\pi t) dt \Big)
\]
and 
\[
B=2e^K\Big(1-\frac{1}{4}e^{2K}\Big)^{-1/4}.
\]
\label{wright}
\end{theorem}

Let $w_0$ be the (unique) $w \in (0,1)$ at which the maximum of $w^{-1}\int_0^w \log \sin(\pi t) dt$
occurs; doing integration by parts one can show that $w_0$ is the unique zero $w \in (0,1)$ of $\int_0^w t\cot(\pi t) dt$. We compute that $w_0=0.7912265710\ldots$, from which we get
 $K=0.1986176152\ldots$, so $e^K=1.219715476\ldots$ and 
$B=2.740222990\ldots$.

The constant $K$ in Theorem \ref{wright} is defined using the $\int_0^w \log \sin(\pi t) dt$,
and in the proof of Theorem \ref{bigpoint} we deal with  $\int_0^{\frac{3\pi}{4n}} \log \sin xdx$.
Milnor in the appendix to \cite{MR634431} shows how to use the integrals
$-\int_0^\theta \log|2\sin u|du$ to compute hyperbolic volumes.

Using the fact that $\norm{P_n}_\infty \leq \norm{\widehat{P_n}}_1 \leq
(N+1) \norm{\widehat{P_n}}_\infty$, one can show using Theorem \ref{wright} that $\lim_{n \to \infty} \norm{P_n}_{\infty}^{1/n}=e^K$.
Freiman and Halberstam \cite{MR937933} give a different proof of this.

One can show for a fixed $f \in L^\infty(\mathbb{T})$ that $\norm{f}_p$ is an increasing function
of $p$. So Theorem \ref{wright} gives us  for $1 \leq p \leq \infty$ that 
\[
\norm{P_n}_p \leq \norm{P_n}_\infty \leq \norm{\widehat{P_n}}_1 \leq
(N+1) \norm{\widehat{P_n}}_\infty \sim \frac{nB e^{Kn}}{2}.
\]
On the other hand,
\[
\norm{P_n}_p \geq \norm{P_n}_1 \geq \norm{\widehat{P_n}}_\infty \sim \frac{B e^{Kn}}{n}.
\]

In fact, the method of Wright's proof can be used to estimate the $L^p$ norms of $P_n$ for $1 \leq p
\leq \infty$. The following is an outline of how to prove this estimate.

\begin{theorem}
Let $C^2=-\frac{1}{2} \frac{\pi}{w_0} \cot(\pi w_0)$, $C>0$, where $w_0$ is the unique $w \in (0,1)$ at which the maximum of $w^{-1}\int_0^w \log \sin(\pi t) dt$ occurs, and take $K$ and $B$ as defined in
Theorem \ref{wright}.
For each $1 \leq p < \infty$ we have
\[
\norm{P_n}_p \sim \Big(2 \frac{\pi^{1/2}}{C p^{1/2} n^{3/2}}\Big)^{1/p} e^{Kn} BC
\Big( \frac{n}{4\pi}\Big)^{1/2},
\]
and for $p=\infty$ we have
\[
\norm{P_n}_\infty \sim e^{Kn} BC\Big( \frac{n}{4\pi}\Big)^{1/2}.
\]
\label{Lpestimate}
\end{theorem}
\begin{proof}[Proof sketch]
Take $1 \leq p < \infty$.
Let $\Pi_n(\theta)=\prod_{k=1}^n 2 |\sin(\pi k \theta)|$. 
Let
$J=[\theta_0-\gamma,\theta_0+\gamma]$, where $\theta_0=\frac{w_0}{n}$ and $\gamma=n^{-4/3}$.
(In fact, the proof works more smoothly if one  chooses $\gamma$ so that  the exponent of $n$ is strictly between $-\frac{3}{2}$ and $-\frac{4}{3}$, 
say their arithmetic mean $-\frac{17}{12}$.)
Using Sudler's work
\cite{MR0163890}, Wright shows that if $\theta \in [0,\frac{1}{2}] \setminus J$ then $\Pi_n(\theta)=o\Big(\frac{e^{Kn}}{n}\Big)$. 
Using the Euler-Maclaurin summation formula, Wright gets an approximation to $\Pi_n(\theta)$ in the interval
$[\frac{1}{2n},\frac{1+w_0}{2n}]$. 
Then using this approximation we can show that
\[
\int_J \Pi_n(\theta)^p d\theta=\frac{\pi^{1/2}}{C p^{1/2} n^{3/2}} \Big( e^{Kn} BC
\Big( \frac{n}{4\pi}\Big)^{1/2}\Big)^p\Big(1+o(1)\Big),
\]
where $C^2=-\frac{1}{2} \frac{\pi}{w_0} \cot(\pi w_0)$, $C>0$. One can compute that $C=1.606193491\ldots$. It follows that
\begin{eqnarray*}
\norm{P_n}_p&=&\Big( 2 \int_0^{1/2} \Pi_n(\theta)^p d\theta \Big)^{1/p}\\
&=&\Big( 2 \int_J \Pi_n(\theta)^p + 2\int_{[0,\frac{1}{2}]\setminus J} \Pi_n(\theta)^p d\theta \Big)^{1/p}\\
&=&\Bigg( 2 \frac{\pi^{1/2}}{C p^{1/2} n^{3/2}} \Big( e^{Kn} BC
\Big( \frac{n}{4\pi}\Big)^{1/2}\Big)^p\Big(1+o(1)\Big) + o(e^{pKn}n^{-p}) \Bigg)^{1/p}\\
&=&\Big( 2 \frac{\pi^{1/2}}{C p^{1/2} n^{3/2}} \Big( e^{Kn} BC
\Big( \frac{n}{4\pi}\Big)^{1/2}\Big)^p+o(e^{pKn} n^{\frac{p-3}{2}})+ o(e^{pKn}n^{-p}) \Big)^{1/p}\\
&=&\Big( 2 \frac{\pi^{1/2}}{C p^{1/2} n^{3/2}} \Big( e^{Kn} BC
\Big( \frac{n}{4\pi}\Big)^{1/2}\Big)^p+o(e^{pKn} n^{\frac{p-3}{2}})\Big)^{1/p}\\
&=&\Big(2 \frac{\pi^{1/2}}{C p^{1/2} n^{3/2}}\Big)^{1/p} e^{Kn} BC\Big( \frac{n}{4\pi}\Big)^{1/2}(1+o(1))^{1/p}\\
&=&\Big(2 \frac{\pi^{1/2}}{C p^{1/2} n^{3/2}}\Big)^{1/p} e^{Kn} BC\Big( \frac{n}{4\pi}\Big)^{1/2}(1+o(1)).
\end{eqnarray*}
In summary we have 
\[
\norm{P_n}_p \sim \Big(2 \frac{\pi^{1/2}}{C p^{1/2} n^{3/2}}\Big)^{1/p} e^{Kn} BC
\Big( \frac{n}{4\pi}\Big)^{1/2}.
\]

From the approximation Wright gets for $\Pi_n(\theta)$, we obtain
\[
\Pi_n(\theta_0)=e^{Kn} BC\Big( \frac{n}{4\pi}\Big)^{1/2}\cdot (1+o(1)),
\]
and then that $\norm{P_n}_\infty \sim e^{Kn} BC\Big( \frac{n}{4\pi}\Big)^{1/2}$.
\end{proof}

We have only sketched the proof of Theorem \ref{Lpestimate}, and to make this estimate plausible to
a reader who doesn't want to read Wright \cite{MR0163891} and Sudler \cite{MR0163890}, we
show in Figure \ref{L1plot} a plot of $\frac{\norm{P_n}_1}{e^{Kn} n^{-1}}$ for $n=1,\ldots,400$ and in Figure
\ref{L2plot} a plot of $\frac{\norm{P_n}_2}{e^{Kn} n^{-1/4}}$ for $n=1,\ldots,400$. We have from Theorem \ref{Lpestimate} that
\[
\norm{P_n}_1 \sim B e^{Kn} n^{-1}=2.740222990\ldots \cdot e^{Kn} n^{-1}
\]
and
\[
\norm{P_n}_2 \sim 2^{-3/4} \pi^{-1/4} BC^{1/2} e^{Kn} n^{-1/4}=1.551046691\ldots \cdot e^{Kn}n^{-1/4}.
\]

\begin{figure}
\includegraphics{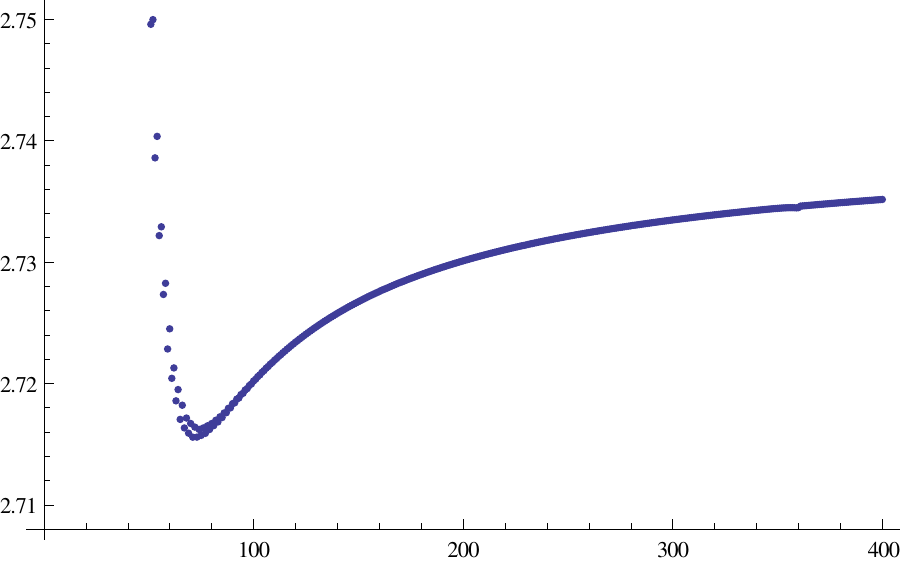}
\caption{$\frac{\norm{P_n}_1}{e^{nK} n^{-1}}$, for $n=1,\ldots,400$}
\label{L1plot}
\end{figure}

\begin{figure}
\includegraphics{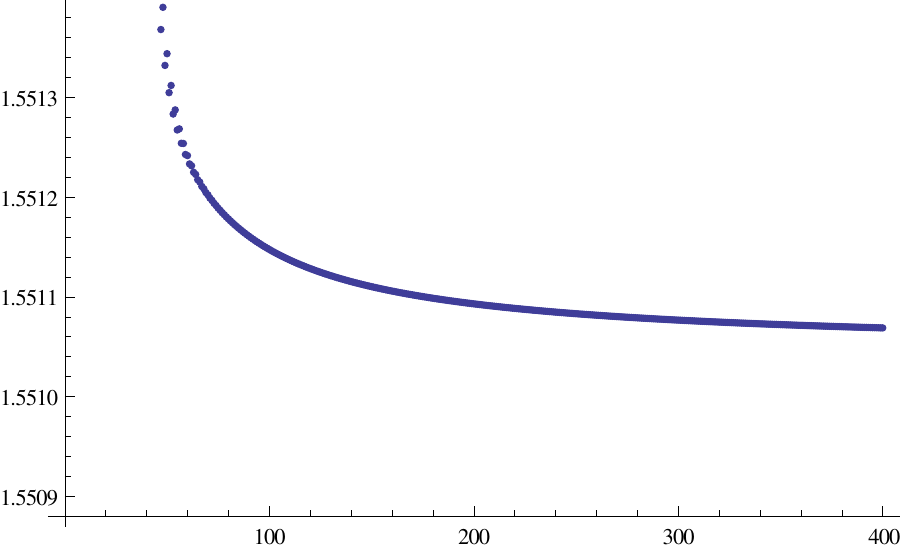}
\caption{$\frac{\norm{P_n}_2}{e^{nK} n^{-1/4}}$, for $n=1,\ldots,400$}
\label{L2plot}
\end{figure}

Using the pentagonal number theorem we can deduce that $\norm{P_n}_1 \to \infty$ as $n \to \infty$
from a general result on exponential sums. Littlewood's conjecture, proved in
\cite{MR621019}, is that
there is a constant $H$ such that
if the first $M$ nonzero Fourier coefficients of an $L^1$ function $f$
each has absolute value $\geq 1$, then
$\norm{f}_1 \geq H \log M$.
The case of the Dirichlet kernel shows us that $H \leq \frac{4}{\pi^2}$,
since $\norm{D_n}_1=\frac{4}{\pi^2}\log n +O(1)$.
Of course all the nonzero Fourier coefficients of $P_n$ have absolute value $\geq 1$, and one can
show using the pentagonal number theorem  that $P_n$ has $\geq \frac{3}{2} \sqrt{n}$ nonzero Fourier coefficients  with absolute value $\geq 1$, hence
\[
\norm{P_n}_1 \geq H \log\Big(\frac{3}{2} \sqrt{n} \Big).
\]

The $L^\infty$ norm of $\prod_{k=1}^n \sin(k\theta)$ is discussed by Carley and Li \cite{MR1790921}. They observe that the maximum of 
$\prod_{k=1}^n \sin(k\theta)$
occurs
around $\theta=\frac{3\pi}{4n}$. Using the Euler-Maclaurin summation formula, they 
show that
\[
\prod_{k=1}^n \sin\Big(\frac{3\pi k}{4n} \Big) \geq 
C \sqrt{n}\exp\Big( -\frac{5}{6}n\log 2 \Big),
\]
for some $C>0$. Thus
\begin{equation}
|P_n\Big(\frac{3\pi}{2n}\Big)| \geq C \sqrt{n} \exp\Big(\frac{1}{6}n \log 2 \Big).
\label{carley}
\end{equation}

We shall improve on the lower bound given in \eqref{carley}. Let $A=\frac{2G}{3\pi}$,
where
\[
G=\sum_{n=0}^\infty \frac{1}{(2n+1)^2} (-1)^n= 0.9159655942\ldots
\]
is Catalan's constant.

\begin{theorem}
For some $C_0>0$, 
\[
\Big|P_n\Big(\frac{3\pi}{2n}\Big)\Big| \leq n^{C_0}e^{An}
\]
and
\[
\Big|P_n\Big(\frac{3\pi}{2n}\Big)\Big| \geq n^{-C_0} e^{An}.
\]
\label{bigpoint}
\end{theorem}
\begin{proof}
Let $f(x)=\log|\sin x|$. Let $l=\lfloor \frac{2n}{3} \rfloor$. We have
\[
\begin{split}
&\Big| \sum_{k=1}^n f\Big(\frac{3\pi k}{4n}\Big) \cdot \frac{3\pi}{4n} - 
\int_0^{\frac{3\pi}{4}} f(x) dx \Big|\\
\leq&\sum_{k=1}^n \Big| f\Big(\frac{3\pi k}{4n}\Big)\cdot \frac{3\pi}{4n}
-\int_{(k-1)\frac{3\pi}{4n}}^{k \frac{3\pi}{4n}} f(x)dx \Big|\\
\leq& f\Big(\frac{3\pi}{4n}\Big) \frac{3\pi}{4n} - \int_0^{\frac{3\pi}{4n}} f(x)dx \\
&+\sum_{k=2}^l \Big( f\Big(k\frac{3\pi}{4n}\Big)-f\Big((k-1)\frac{3\pi}{4n}\Big)\Big) \frac{3\pi}{4n}\\
&+\Big(f(\frac{\pi}{2})-f\Big(l\frac{3\pi}{4n} \Big)\Big)\frac{3\pi}{4n}+\Big(f(\frac{\pi}{2})-f\Big((l+1)\frac{3\pi}{4n} \Big)\Big)\frac{3\pi}{4n}\\
&+\sum_{k=l+2}^n \Big( f\Big((k-1)\frac{3\pi}{4n}\Big)-f\Big(k\frac{3\pi}{4n}\Big)\Big) \frac{3\pi}{4n}.
\end{split}
\]
We will estimate these lines separately. For the first line, because $\sin x \leq x$ for all $x \geq 0$ and
because $\sin x \geq \frac{2}{\pi} x$ for
$x \in [0,\frac{\pi}{2}]$,
\begin{eqnarray*}
\frac{3\pi}{4n} \log \sin \frac{3\pi}{4n}  - \int_0^{\frac{3\pi}{4n}} \log \sin xdx
&\leq& \frac{3\pi}{4n} \log \frac{3\pi}{4n}  - \int_0^{\frac{3\pi}{4n}} \log \frac{2}{\pi}x dx\\
&=&\frac{3\pi}{4n} \log \frac{3\pi}{4n}-\frac{3\pi}{4n} \log \frac{2}{\pi}
-\frac{3\pi}{4n} \log \frac{3\pi}{4n} + \frac{3\pi}{4n}\\
&=&\frac{3\pi}{4n}\Big(1-\log \frac{2}{\pi}\Big)\\
&=&O\Big(\frac{1}{n}\Big).
\end{eqnarray*}

For the second line, because $f'(x)=\cot x$ we have
\begin{eqnarray*}
\sum_{k=2}^l \Big( f\Big(k\frac{3\pi}{4n}\Big)-f\Big((k-1)\frac{3\pi}{4n}\Big)\Big) \frac{3\pi}{4n}&=&
\Big( f\Big(l\frac{3\pi}{4n}\Big)-f\Big(\frac{3\pi}{4n}\Big)\Big) \frac{3\pi}{4n} \\
 &=&\frac{3\pi}{4n} \int_{\frac{3\pi}{4n}}^{l\frac{3\pi}{4n}} \cot x dx\\
 &\leq& \frac{3\pi}{4n} \int_{\frac{3\pi}{4n}}^{\frac{\pi}{2}} \cot x dx\\
 &=& \frac{3\pi}{4n} \Big( f\Big(\frac{\pi}{2} \Big) - f\Big(\frac{3\pi}{4n}\Big)\Big)\\
 &=&-\frac{3\pi}{4n} \log \sin  \frac{3\pi}{4n}\\
 &=&O\Big(\frac{\log n}{n} \Big).
\end{eqnarray*}

For the third line,  $|f\Big((l+1)\frac{3\pi}{4n} \Big)| \leq |f\Big(l\frac{3\pi}{4n} \Big)|$. Moreover,
$\lfloor \frac{2n}{3} \rfloor \geq \frac{n}{2}$ for $n \geq 2$, so $|f\Big(l\frac{3\pi}{4n} \Big)| \leq |\log \sin \frac{3\pi}{8}|$. Therefore the third line is
$O(\frac{1}{n})$.

For the fourth line,  because $f'(x)=\cot x$ we have
\begin{eqnarray*}
\sum_{k=l+2}^n \Big( f\Big((k-1)\frac{3\pi}{4n}\Big)-f\Big(k\frac{3\pi}{4n}\Big)\Big) \frac{3\pi}{4n}&=&
-\frac{3\pi}{4n}\Big(f\Big(\frac{3\pi}{4}\Big)-  f\Big((l+1)\frac{3\pi}{4n}\Big) \Big)\\
&=&-\frac{3\pi}{4n} \int_{(l+1)\frac{3\pi}{4n}}^{\frac{3\pi}{4}} \cot x dx\\
&\leq&-\frac{3\pi}{4n} \int_{\frac{\pi}{2}}^{\frac{3\pi}{4}} \cot x dx\\
&=&O\Big(\frac{1}{n} \Big).
\end{eqnarray*}

The sum of the four lines is $O\Big( \frac{\log n}{n}\Big)$, and
thus there is some $C_0>0$ such that 
\[
\Big| \sum_{k=1}^n \log\sin\Big(\frac{3\pi k}{4n}\Big) - 
\frac{4n}{3\pi} \int_0^{\frac{3\pi}{4}} \log\sin x dx \Big| \leq C_0\log n.
\]

One can check that $\log|\sin x|$ has the Fourier series
\[
-\log 2 - \sum_{n=1}^\infty \frac{1}{n} (1+(-1)^n) \cos(nx).
\]
If $f \in L^1(\mathbb{T})$ has the Fourier series $\sum a_k e^{ikx}$, then
$\int_a^b f(x) dx = \sum a_k \int_a^b e^{ikx} dx$ \cite[\S 13.5]{titchmarsh}.  
For $f(x)=\log |\sin x|$, $a=0$, and $b=\frac{3\pi}{4}$, we have
\begin{eqnarray*}
\int_0^{\frac{3\pi}{4}} \log \sin x dx&=&-\frac{3\pi \log 2}{4} -\sum_{n=1}^\infty \frac{1}{n^2}(1+(-1)^n)
\sin \frac{3\pi n}{4}\\
&=&-\frac{3\pi \log 2}{4} -\sum_{n=1}^\infty \frac{1}{(2n)^2} \cdot 2 \sin \frac{3\pi \cdot 2n}{4}\\
&=&-\frac{3\pi \log 2}{4} - \frac{1}{2}\sum_{n=1}^\infty \frac{1}{n^2} \sin \frac{3\pi n}{2}\\
&=&-\frac{3\pi \log 2}{4} - \frac{1}{2}\sum_{n=1}^\infty \frac{1}{(2n+1)^2} \sin \frac{3\pi \cdot (2n+1)}{2}\\
&=&-\frac{3\pi \log 2}{4} - \frac{1}{2}\sum_{n=1}^\infty \frac{1}{(2n+1)^2} (-1)^{n+1}\\
&=&-\frac{3\pi \log 2}{4} + \frac{G}{2}.
\end{eqnarray*}
Therefore
\[
\Big| \sum_{k=1}^n \log\sin\Big(\frac{3\pi k}{4n}\Big) - 
(A-\log 2)n \Big| \leq C_0\log n
\]
for $A=\frac{2G}{3\pi}$.
Taking exponentials, it follows that
\[
\prod_{k=1}^n \sin\Big(\frac{3\pi k}{4n}\Big) \leq n^{C_0} e^{(A-\log 2)n}
\]
and
\[
\prod_{k=1}^n \sin\Big(\frac{3\pi k}{4n}\Big) \geq n^{-C_0} e^{(A-\log 2)n}.
\]
Thus by \eqref{sineproduct} we get
$|P_n\Big(\frac{3\pi}{2n}\Big)| \leq n^{C_0}e^{An}$
and $|P_n\Big(\frac{3\pi}{2n}\Big)| \geq n^{-C_0}e^{An}$.
\end{proof}

This shows that $\norm{P_n}_\infty \geq n^{-C_0} e^{An}$.
One can compute that $e^A=1.214550362\ldots$.

In the following theorem, we use the fact that $|P_n(\theta)|$ is large at $\theta=\frac{3\pi}{2n}$ and is $0$ at $\theta=0$
to get a lower bound on the $L^2$ norm of $P_n$. It is worse than the asymptotic formula that
we get from Theorem \ref{Lpestimate}, but its proof doesn't use the results of Wright \cite{MR0163891} and Sudler \cite{MR0163890}.

\begin{theorem} 
We have
\[
\norm{P_n}_2 \geq \frac{n^{-C_0} e^{An}}{\sqrt{2.2n(n+1)}}.
\]
\label{bigsmall}
\end{theorem}
\begin{proof}
Davenport and Halberstam \cite{MR0197427} prove the following. 
Let $a_{-N},\ldots,a_N$ be complex numbers, and define
\[
S(x)=\sum_{k=-N}^N a_k e^{2\pi ikx}.
\]
For $R \geq 2$, let $x_1,\ldots,x_R$ be real numbers and put
$\delta=\min_{j \neq k} \norm{x_j-x_k}$, where $\norm{\theta}$
is the distance from $\theta$ to the nearest integer, e.g. $\norm{\frac{1}{10}}=\frac{1}{10}$ and $\norm{-\frac{7}{10}}=\frac{3}{10}$. 
We have that
\[
\sum_{r=1}^R |S(x_r)|^2 \leq 2.2 \max(\delta^{-1},2N) \sum_{k=-N}^N |a_k|^2.
\] 

Take $N=\frac{n(n+1)}{2}$, and $a_k=\widehat{P_n}(k)$. This gives $S(\frac{x}{2\pi})=P_n(x)$. 
Let $R=2$, $x_1=0$ and $x_2=\frac{3}{4n}$. Therefore $\delta=\frac{3}{4n}$ 
and so $\max(\delta^{-1},2N)=\max(\frac{4n}{3},n(n+1))=n(n+1)$. 
Then from Davenport and Halberstam's result we have that
\[
|S(0)|^2+|S\Big(\frac{3}{4n}\Big)|^2 \leq 2.2n(n+1) \sum_{k=-N}^N |a_k|^2=
2.2n(n+1) \norm{\widehat{P_n}}_2^2.
\]
Of course $S(0)=0$. By Parseval's theorem, $\norm{\widehat{P_n}}_2=\norm{P_n}_2$. So
\[
|P_n\Big(\frac{3\pi}{2n}\Big)|^2 \leq 2.2n(n+1) \norm{P_n}_2^2.
\]

We  proved in Theorem \ref{bigpoint} that
$|P_n\Big(\frac{3\pi}{2n}\Big) | \geq n^{-C_0} e^{An}$.
This gives us
\[
\frac{n^{-C_0} e^{An}}{\sqrt{2.2n(n+1)}} \leq \norm{P_n}_2.
\]\end{proof}

Lubinsky  \cite[Theorem 1.1]{MR1684685} proves that if $\epsilon>0$, then for almost all $\theta$ we have
\[
|\log |P_n(\theta)|| = O((\log n)(\log \log n)^{1+\epsilon}),
\]
but that this is false if $\epsilon=0$.
If $\theta$ has bounded partial quotients, Lubinsky shows that $\log |P_n(\theta)| = O(\log n)$ \cite[Theorem 1.3]{MR1684685}.
However, almost all $\theta$ do not
have a continued fraction expansion with bounded partial quotients \cite[p. 166, Theorem 196]{MR2445243}.

\section{$Q_n$}
\label{Qnsection}
One can  see that the Fourier coefficient $\widehat{Q_n}(j)$ is equal to the number of ways to write $j$ as a sum of distinct positive integers each $\leq n$. 
For example, the partitions of $9$ into distinct parts each $\leq 6$ are: $1+2+6,1+3+5,2+3+4,2+7,3+6,4+5$, and thus
$\widehat{Q}_7(9)=6$.

Various results have been proved about the number of partitions of $j$ as a sum of integers each $\geq n$ and the number of partitions of $j$ as a sum of distinct integers each $\geq n$ for $n$ small
relative to $j$, e.g. Szekeres \cite{MR0057279}, Freiman and Pitman
\cite{MR1297011}, and Mosaki \cite{MR2477513}.

By \eqref{cosineproduct},
we can express $Q_n(\theta)$ using $\prod_{k=1}^n \cos\Big(\frac{k\theta}{2} \Big)$.
The product  $\prod_{k=1}^n \cos (k\theta)$ has
the following probabilistic interpretation. Let $X_k$ be independent Bernoulli $\pm 1$ random variables. One can check that the characteristic function of $\sum_{k=1}^n kX_k$ is $\prod_{k=1}^n \cos (k\theta)$. Unfortunately, to use the central limit theorem we would first have to normalize the sum by dividing it by $n^{3/2}$, and the characteristic function of $\sum_{k=1}^n \frac{k}{n^{3/2}} X_k$ is $\prod_{k=1}^n \cos \frac{k\theta}{n^{3/2}}$, not $\prod_{k=1}^n \cos (k\theta)$.

We see from \eqref{cosineproduct} that $|Q_n(\theta)|=\prod_{k=1}^n 2|\cos\Big(\frac{k\theta}{2} \Big)|$. In Figure \ref{cosinefigure} we plot $\prod_{k=1}^{10} 2|\cos (k\theta)|$ for $0 \leq \theta \leq \frac{\pi}{2}$; in this plot the ordinate of $0$ is $1024$.

\begin{figure}
\includegraphics{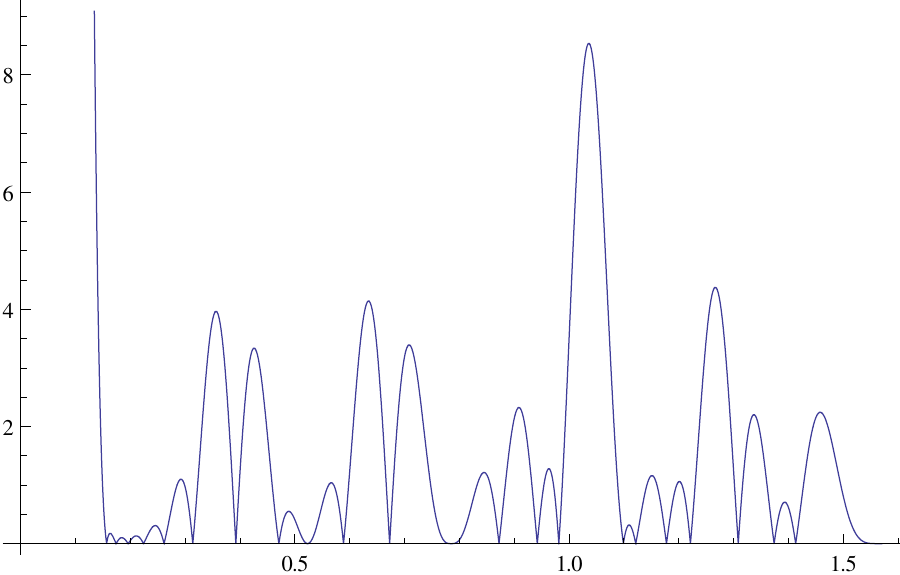}
\caption{$\prod_{k=1}^{10} 2|\cos (k\theta)|$ for $0 \leq \theta \leq \frac{\pi}{2}$}
\label{cosinefigure}
\end{figure}

Of course $Q_n(0)=2^n$, so
$\norm{Q_n}_\infty=2^n$. Aside from $\theta=0$ we can  explicitly evaluate $Q_n(\theta)$ for certain other $\theta$. For any $h$ with $\gcd(n+1,h)=1$, we have $z^{n+1}-1=\prod_{k=1}^{n+1} (z-e^{\frac{2\pi ihk}{n+1}})$.
Since $z^{n+1}-1=(z-1)(z^n+\cdots+z+1)$, we get
$z^n+\cdots+z+1=\prod_{k=1}^n (z-e^{\frac{2\pi ihk}{n+1}})$, and setting $z=-1$ yields 
\[
 Q_n\Big(\frac{2\pi h}{n+1}\Big)=\frac{1+(-1)^n}{2}
\]
for each $h$
with $\gcd(n+1,h)=1$.

For all
$1 \leq p \leq \infty$ we have $\norm{Q_n}_p \leq \norm{Q_n}_\infty =2^n$. On the other hand,
let $1 \leq p \leq q \leq \infty$. One can show that there is some $C>0$ such that
if $f$ satisfies $\hat{f}(j)=0$ for $|j|>N$ then $\norm{f}_q \leq CN^{\frac{1}{p}-\frac{1}{q}} \norm{f}_p$ \cite[p. 123, Exercise 1.8]{katznelson}. (In fact one can take $C=5$.)
Since $\norm{Q_n}_\infty=2^n$, we get for $1 \leq p \leq \infty$ that
$\norm{Q_n}_p \geq \frac{1}{C} 2^n N^{-1/p}$.

We can do better than this. Following Wright's method in the proof
of Theorem \ref{wright}, which we used in our sketch of the proof of Theorem \ref{Lpestimate}, we
get in the following theorem an asymptotic formula for $\norm{Q_n}_p$.

\begin{theorem}
For $1 \leq p < \infty$ we have
\[
\norm{Q_n}_p \sim \Big(\frac{6}{p\pi} \Big)^{\frac{1}{2p}} 2^n n^{-\frac{3}{2p}} .
\]
\label{Qnestimate}
\end{theorem}
\begin{proof}
Let $\Psi_n(\theta)=\prod_{k=1}^n 2|\cos (\pi k\theta)|$. We can check that
\[
\norm{Q_n}_p=\Big( 2 \int_0^{1/2} \Psi_n(\theta)^p d\theta \Big)^{1/p}.
\]
Let $\gamma=n^{-4/3}$. We shall estimate $\Psi_n(\theta)$ separately for $0 \leq \theta \leq \gamma$
and for $\gamma \leq \theta \leq \frac{1}{2}$.

Let $0 \leq \theta \leq \gamma$. We define $F(y)$, depending on $\theta$, by $F(y)=\log \cos(\pi \theta y)$. Then
\[
\log \Psi_n(\theta)=n \log 2 + \sum_{k=1}^n \log \cos (\pi k\theta) = n\log 2 + \sum_{k=1}^n F(k).
\]
By the Euler-Maclaurin summation formula \cite[p. 303, Eq. 7.2.4]{dieudonne} we have
\[
\sum_{k=1}^n F(k) = \int_0^n F(y) dy+\underbrace{\frac{1}{2}F(n)+\frac{1}{2}F(1)+\frac{1}{12}F'(n)-\frac{1}{12}F'(1)-\int_0^1 F(y) dy}_{M_n} + R_n
\]
where $|R_n| \leq \frac{2}{(2\pi)^2} \int_1^n |F'''(y)| dy$. First, doing a change of variables,
because $\theta \leq \gamma=n^{-4/3}$ and
because $\log \cos x=-\frac{x^2}{2}+O(x^4)$,
\begin{eqnarray*}
\int_0^n F(y)dy&=&\frac{1}{\theta} \int_0^{n\theta} \log \cos(\pi z)dz\\ 
&=&\frac{1}{\theta} \int_0^{n\theta}\Big( -\frac{\pi^2z^2}{2}+O(z^4) \Big) dz\\
&=&-\frac{\pi^2}{6} n^3 \theta^2+O(n^5 \theta^4)\\
&=&-\frac{\pi^2}{6} n^3 \theta^2+O(n^{-1/3}).
\end{eqnarray*}

Second, using $\theta \leq \gamma=n^{-4/3}$,  $\log \cos x=-\frac{x^2}{2}+O(x^4)$, and $\tan x=O(x)$,
we have
\begin{eqnarray*}
M_n&=&\frac{1}{2}\log \cos(\pi \theta n)+\frac{1}{2}\log \cos(\pi \theta)-\frac{\pi \theta}{12}\tan(\pi \theta n)+\frac{\pi \theta}{12}\tan(\pi \theta)\\
&&-\int_0^1 \log\cos(\pi \theta y) dy\\
&=&O(n^{-2/3})+O(n^{-8/3})+O(n^{-5/3})+O(n^{-8/3})+O(n^{-8/3})\\
&=&O(n^{-2/3}).
\end{eqnarray*}
Third, $F'''(y)=-2\pi^3 \theta^3 \sec^2(\pi \theta y)\tan(\pi \theta y)$, which yields $|R_n|=O(n^{-10/3})$. Putting these three pieces together gives
\[
\log \Psi_n(\theta)=n \log 2-\frac{\pi^2}{6} n^3 \theta^2+O(n^{-1/3})
\] 
and thus
\[
\Psi_n(\theta)=2^n \exp\Big(-\frac{\pi^2}{6} n^3 \theta^2 \Big)\exp(O(n^{-1/3}))=
2^n \exp\Big(-\frac{\pi^2}{6} n^3 \theta^2 \Big)(1+O(n^{-1/3})).
\]

Therefore, making the change of variables $\phi=\sqrt{\frac{p}{6}} \pi n^{3/2} \theta$ and
because $\int_0^V e^{-\phi^2} d\phi \sim \frac{\sqrt{\pi}}{2}-\frac{\exp(-V^2)}{2V}$ as $V \to \infty$ \cite[p. 97, Eq. 10.8.4]{dieudonne},
\begin{eqnarray*}
\int_0^\gamma \Psi_n(\theta)^p d\theta&=&2^{pn} \int_0^\gamma \exp\Big(-p\frac{\pi^2}{6} n^3 \theta^2 \Big) d\theta \cdot (1+O(n^{-1/3}))\\
&=&2^{pn} \sqrt{\frac{6}{p}} \pi^{-1} n^{-3/2} \int_0^{n^{1/6} \sqrt{\frac{p}{6}} \pi} e^{-\phi^2}   d\phi \cdot (1+O(n^{-1/3}))\\
&=&2^{pn} \sqrt{\frac{3}{2p\pi}} n^{-3/2} \cdot  (1+O(n^{-1/3})) (1+O(n^{-1/6}))\\
&=&2^{pn} \sqrt{\frac{3}{2p\pi}} n^{-3/2} \cdot  (1+O(n^{-1/6})).
\end{eqnarray*}

Now we bound $\Psi_n(\theta)$ for $\gamma \leq \theta \leq \frac{1}{2}$. We have,
for $\Psi_n(\theta) \neq 0$,
\[
\Psi_n(\theta)=\exp(\log \Psi_n(\theta))=
2^n \exp\Big( \sum_{k=1}^n \log |\cos (\pi k \theta)| \Big).
\]
Using the inequality $\log x \leq x-1$ for $x>0$ and the identity $\cos(2x)=2\cos^2 x-1$, we get
 for all $x$ with $\cos x \neq 0$ that
\[
\log | \cos x| 
=\frac{1}{2} \log(\cos^2 x)
\leq \frac{1}{2} (\cos^2 x-1)
=\frac{1}{4}(-1+\cos(2x)).
\]
Hence, for $\Psi_n(\theta) \neq 0$,  
\[
\Psi_n(\theta) \leq 2^n \exp \Big(\frac{1}{4}\sum_{k=1}^n (-1+\cos(2\pi k \theta))  \Big);
\]
but of course this inequality is true when $\Psi_n(\theta)=0$, hence the inequality is true for all
$\theta$.
Let
\[
H_n(\theta)=\sum_{k=1}^n (-\log 2+\cos(2\pi k \theta)).
\]

We first deal with the interval  $\gamma \leq \theta \leq \frac{1}{2\pi n}$. For $0 \leq x \leq 1$ one has $\cos x \leq 1-\frac{x^2}{2}$ (using the Taylor series for $\cos x$,
which is an alternating series), so for $\gamma \leq \theta \leq \frac{1}{2\pi n}$ we have
\[
H_n(\theta) \leq \sum_{k=1}^n -\frac{(2\pi k\theta)^2}{2}
=-2\pi^2 \theta^2 \sum_{k=1}^n k^2
=-2\pi^2 \theta^2 \frac{2n^3+3n^2+n}{6}
\leq -\frac{2\pi^2 \theta^2 n^3}{3},
\]
so $H_n(\theta) \leq -\frac{2\pi^2 n^{1/3}}{3}$.

We now deal with the interval $\frac{1}{2\pi n} \leq \theta \leq \frac{1}{2n}$. 
Since $-1+\cos x=-2\sin^2(\frac{x}{2})$, we have
\[
H_n(\theta)=-2 \sum_{k=1}^n \sin^2 (\pi k \theta).
\]
Using that $\sin^2 x$ is nondecreasing for $0 \leq x \leq \frac{\pi}{2}$ we have
\begin{eqnarray*}
\Big|  \sum_{k=1}^n \pi \theta \sin^2(\pi k \theta) - \int_0^{\pi n \theta} \sin^2 x dx\Big|&=&\Big| \sum_{k=1}^n\Big(\pi \theta \sin^2(\pi k \theta) - \int_{(k-1)\pi \theta}^{k\pi \theta} \sin^2 x dx \Big)  \Big|\\
&\leq&\sum_{k=1}^n\Big|\pi \theta \sin^2(\pi k \theta) - \int_{(k-1)\pi \theta}^{k\pi \theta} \sin^2 x dx  \Big|\\
&\leq&\sum_{k=1}^n \Big| \pi \theta \sin^2(\pi k \theta) -\pi \theta \sin^2((k-1)\pi \theta) \Big|\\
&=&\pi \theta \sum_{k=1}^n \Big| \sin(\pi \theta) \sin((2k-1) \pi \theta) \Big|\\
&\leq&\pi \theta \sum_{k=1}^n \pi \theta\\
&\leq&\frac{\pi^2}{4n}.
\end{eqnarray*}
Therefore
\[
 \sum_{k=1}^n \pi \theta \sin^2(\pi k \theta) \geq \int_0^{\pi n \theta} \sin^2 x dx - \frac{\pi^2}{4n}.
\]
But $\int_0^{\pi n\theta} \sin^2 x dx \geq \int_0^{1/2} \sin^2 x dx=\frac{1}{4}(1-\sin(1))$, because $\theta \geq \frac{1}{2\pi n}$, so
\[
 \sum_{k=1}^n \sin^2(\pi k \theta)
 \geq
 \frac{1}{4\pi\theta}(1-\sin(1)) -
\frac{\pi}{4n\theta}
\geq \frac{n}{2\pi}(1-\sin(1))-\frac{\pi^2}{2}.
\]
So for $\frac{1}{2\pi n} \leq \theta  \leq \frac{1}{2n}$ we have
\[
H_n(\theta) \leq -\frac{n}{\pi}(1-\sin(1)) + \pi^2.
\]

Finally we deal with the interval $\frac{1}{2n} \leq \theta \leq \frac{1}{2}$.
Using $\cos x=\frac{e^{ix}+e^{-ix}}{2}$, the formula for a finite geometric series, and 
then $\sin x=\frac{e^{ix}-e^{-ix}}{2i}$, one can check that
\[
H_n(\theta)=-n-\frac{1}{2}+\frac{1}{2} \frac{\sin((2n+1)\pi \theta)}{\sin(\pi \theta)}.
\]
For $0 \leq x \leq \frac{\pi}{2}$ we have $\sin x \geq \frac{2}{\pi}x$, so
for $\frac{1}{2n} \leq \theta \leq \frac{1}{2}$ we have
\[
H_n(\theta) \leq -n-\frac{1}{2}+\frac{1}{4\theta}\leq -\frac{n}{2}-\frac{1}{2}.
\]

Putting together the bounds we have for $\gamma \leq \theta \leq \frac{1}{2\pi n}$, $\frac{1}{2\pi n} \leq \theta \leq \frac{1}{2n}$, and $\frac{1}{2n} \leq \theta \leq \frac{1}{2}$, we get
\[
\Psi_n(\theta)= O\Big(2^n \exp\Big(-\frac{\pi^2 n^{1/3}}{6}\Big)\Big).
\]

In summary, we have shown that
\begin{eqnarray*}
2 \int_0^{1/2} \Psi_n(\theta)^p d\theta&=&2^{pn} \sqrt{\frac{6}{p\pi}} n^{-3/2} \cdot  (1+O(n^{-1/6}))+
O\Big(2^{pn} \exp\Big(-p\frac{\pi^2 n^{1/3}}{6}\Big)\Big)\\
&=&2^{pn} \sqrt{\frac{6}{p\pi}} n^{-3/2} \cdot  (1+O(n^{-1/6})).
\end{eqnarray*}
\end{proof}

In the following theorem we prove that $\norm{Q_n}_1=O\Big(\frac{2^n}{\sqrt{n}}\Big)$. This is better than the trivial upper bound $\norm{Q_n}_1 \leq 2^n$, but is worse than 
Theorem \ref{Qnestimate}, according to which we have $\norm{Q_n}_1 \sim \sqrt{\frac{6}{\pi}} \frac{2^n}{n^{3/2}}$. However, the following theorem has a simpler proof.

\begin{theorem} 
We have
\[
\norm{Q_n}_1 = O\Big(\frac{2^n}{\sqrt{n}}\Big).
\]
\label{hoelder}
\end{theorem}
\begin{proof}
By H\"older's inequality,
\[
\norm{\prod_{k=1}^n \cos(k\theta)}_1 \leq \prod_{k=1}^n \norm{\cos (kt)}_n.
\]
For each $k$,
\[
\int_0^{2\pi} |\cos (kt)|^n dt
=4\int_0^{\pi/2} \cos^n t dt.
\]
Let $G_n=\int_0^{\pi/2} \cos^n t dt$. Using integration by parts and induction (doing the even and odd cases separately) one can show that
\[
G_n=\frac{\sqrt{\pi}}{2} \frac{\Gamma(\frac{n+1}{2})}{\Gamma(\frac{n}{2}+1)}.
\]
Then using Stirling's approximation and the fact that $\lim_{n \to \infty} (1-\frac{1}{n})^n=e^{-1}$ we get
\[
G_n \sim \sqrt{\frac{\pi}{2}} \frac{1}{\sqrt{n+2}}.
\]
\end{proof}

Following Pribitkin's  \cite{MR2471621}, which gives an upper bound on
the number of partitions of $j$ with at most
$n$ parts, Bidar \cite{bidar} gives an upper bound on
$\widehat{Q_n}(j)$ involving the dilogarithm function $\Li$. However, take $n$ to be even, and let
$j=\lfloor \frac{n(n+1)}{4} \rfloor$. We  compute that the exponential term in Bidar's upper bound for
$\widehat{Q_n}(j)$ is  $e^{Ln}$, with
\[
L \geq \frac{\pi}{2\sqrt{3}}-\frac{\sqrt{3}}{2\pi} \Li\Big(\exp\Big(-\frac{\pi}{\sqrt{3}}\Big)\Big)
=0.8599790113\ldots.
\]
But $\log 2=0.6931471805\ldots$. Thus here Bidar's bound is worse than the bound
$\widehat{Q_n}(j) \leq \norm{Q_n}_1 \leq \norm{Q_n}_\infty = 2^n$.

In the following theorem we show that for $j$ sufficiently close to $\frac{n(n+1)}{4}$ the Fourier coefficient $\widehat{Q_n}(j)$ is close to $2^n \sqrt{\frac{6}{\pi}} n^{-3/2}$,
and that $\widehat{Q_n}(j)$ is upper bounded by $2^n \sqrt{\frac{6}{\pi}} n^{-3/2}(1+o(1))$ for all $j$, from which we get
$\norm{\widehat{Q_n}}_\infty \sim 2^n \sqrt{\frac{6}{\pi}} n^{-3/2}$. We use the bounds on $\Psi_n(\theta)$ that we established in our
proof of Theorem \ref{Qnestimate}.

\begin{theorem}
We have 
\[
\norm{\widehat{Q_n}}_\infty \sim 2^n \sqrt{\frac{6}{\pi}} n^{-3/2}.
\]
\label{Qninfinity}
\end{theorem}
\begin{proof}
We can check that 
\[
\widehat{Q_n}(j)=2 \int_0^{1/2} \cos(\pi(N-2j)\theta) \prod_{k=1}^n 2 \cos(\pi k\theta) d\theta,
\qquad N=\frac{n(n+1)}{2}.
\]
Following the proof of Theorem \ref{Qnestimate}, with $\Psi_n(\theta)=\prod_{k=1}^n 2|\cos (\pi k\theta)|$ and $\gamma=n^{-4/3}$, we get
\[
\widehat{Q_n}(j) = 2 \int_0^\gamma \cos(\pi(N-2j)\theta) \Psi_n(\theta) d\theta+O\Big(2^n \exp\Big(-\frac{\pi^2 n^{1/3}}{6}\Big)\Big).
\]

We have from our proof of Theorem \ref{Qnestimate} that
\[
\int_0^\gamma \Psi_n(\theta)d\theta
=2^n \sqrt{\frac{3}{2\pi}} n^{-3/2}\cdot (1+O(n^{-1/6}).
\]
Using this and the inequality $\cos(x) \geq 1-\frac{x^2}{2}$ for $0 \leq x \leq 1$,
we have for $|N-2j|=o(n^{4/3})$ that
\begin{eqnarray*}
\widehat{Q_n}(j)&=&2\int_0^\gamma \Psi_n(\theta) d\theta+o\Big(\int_0^\gamma \Psi_n(\theta) d\theta\Big)
+O\Big(2^n \exp\Big(-\frac{\pi^2 n^{1/3}}{6}\Big)\Big)\\
&=&2^n \sqrt{\frac{6}{\pi}} n^{-3/2}(1+o(1)).
\end{eqnarray*}

But by Theorem \ref{Qnestimate} we have 
$\norm{\widehat{Q_n}}_\infty \leq \norm{Q_n}_1 \sim 2^n \sqrt{\frac{6}{\pi}} n^{-3/2}$. It follows that
$\norm{\widehat{Q_n}}_\infty \sim 2^n \sqrt{\frac{6}{\pi}} n^{-3/2}$.
\end{proof}

In the above proof we showed that $\widehat{Q_n}(j)$ is $2^n \sqrt{\frac{6}{\pi}} n^{-3/2}(1+o(1))$ for 
$|N-2j|=o(n^{4/3})$ and that for other $j$, $\widehat{Q_n}(j)$ is upper bounded by
$2^n \sqrt{\frac{6}{\pi}} n^{-3/2}(1+o(1))$, but we didn't establish whether $\widehat{Q_n}(j)$ is 
close
to $2^n \sqrt{\frac{6}{\pi}} n^{-3/2}$ for other $j$ or is substantially smaller. Generally, a sequence $a_0,\ldots,a_N$ is said to be
{\em symmetric} if $a_k=a_{N-k}$ for all $0 \leq k \leq N$, and is said to be {\em unimodal} if there is some $m$
such that $a_0 \leq a_1 \leq \cdots \leq a_m$ and $a_N \leq a_{N-1} \leq \cdots a_m$.
If $a_0,\ldots,a_N$ is symmetric and unimodal then for $m=\lfloor \frac{N}{2} \rfloor$, 
the term $a_m$ is equal to the maximum of the sequence.
For $N=\frac{n(n+1)}{2}$, 
there is a bijection between the set of partitions of $j$ into distinct parts each $\leq n$ and the set
of partitions of $N-j$ into distinct parts each $\leq n$: for each partition we take the positive integers $\leq n$ not in this partition. Thus  $\widehat{Q_n}(j)=\widehat{Q_n}(N-j)$, i.e.  the sequence $\widehat{Q_n}(j)$ is symmetric. 
Hughes and Van der Jeugt \cite{MR1136915} show using the representation theory
of Lie algebras that 
the sequence $\widehat{Q_n}(j)$ is unimodal, and survey how to use these methods to prove the unimodality of other sequences. 
The unimodality of $\widehat{Q_n}(j)$ can be also be proved without using Lie algebraic methods \cite{MR656013}.

\section{Conclusions}
\label{conclusion}
It remains to determine the asymptotic behavior of the $\ell^p$ norms of $\widehat{P_n}$ and 
$\widehat{Q_n}$ for $1 \leq p <\infty$.
Let $N=\frac{n(n+1)}{2}$ and let $L=\frac{(2N-k)w_0}{n}-\frac{1}{4}n$, 
with $w_0$ as defined following Theorem \ref{wright}. Wright's proof  \cite{MR0163891} of our Theorem \ref{wright} shows that
if $k=\frac{N}{2}+o(n^{3/2})$ then 
\[
\widehat{P_n}(k)=\frac{Be^{Kn}}{n} \cos(2\pi L) + o\Big(\frac{e^{Kn}}{n} \Big),
\]
with $K$ and $B$  as defined in Theorem \ref{wright}. 
Furthermore, Wright \cite{MR0163892}  proves a result that specializes to the following.
Take $C$ as defined in Theorem \ref{Lpestimate}.
 If
$m=k-\frac{N}{2}=o(n^{5/3})$ then
\[
\widehat{P}_n(k)=\frac{B}{n} \exp\Big(Kn-\frac{\pi^2 m^2}{C^2 n^2}\Big)\Big(\cos\Big(\frac{n \pi}{2}+2\pi mn^{-1} w_0\Big)+o(1)\Big).
\]

\begin{figure}
\includegraphics{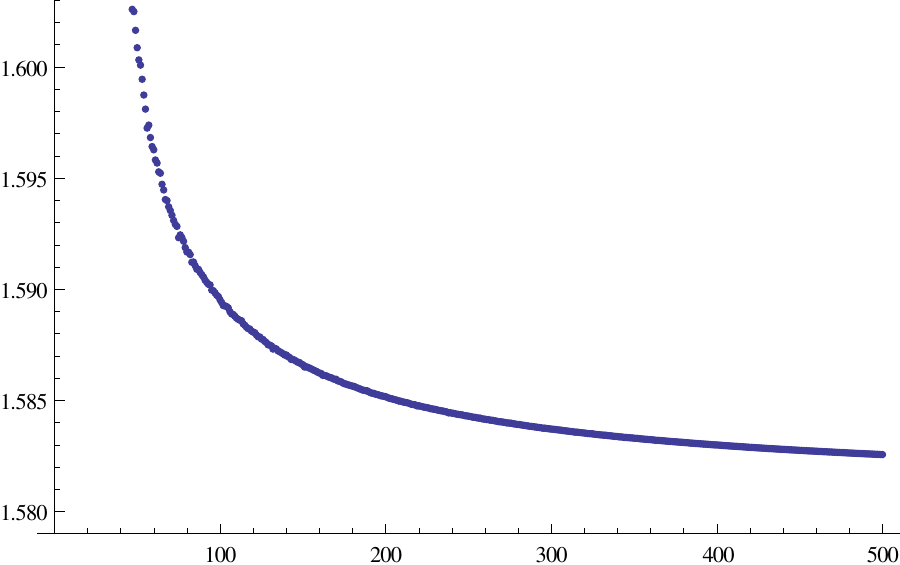}
\caption{$\frac{\norm{\widehat{P_n}}_1}{e^{Kn} n^{1/2}}$, for $n=1,\ldots,500$}
\label{ell1Pn1to500}
\end{figure}

\begin{figure}
\includegraphics{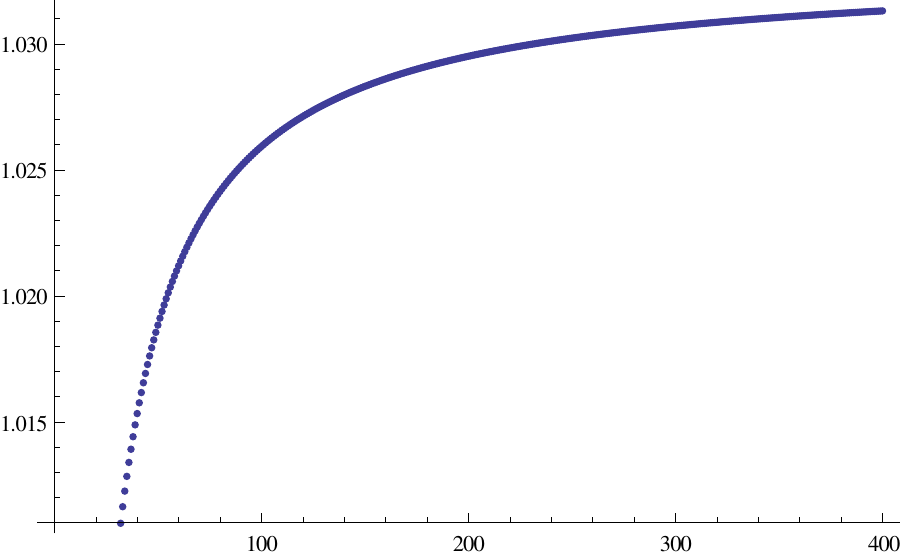}
\caption{$\frac{\norm{\widehat{Q_n}}_3}{2^n n^{-1}}$, for $n=1,\ldots,400$}
\label{ell3Qn1to400}
\end{figure}

If $n^{3/2}$ of the Fourier coefficients of $P_n$ have magnitude on the order of $\frac{e^{Kn}}{n}$
and the other Fourier coefficients of $P_n$ are relatively negligible, then
$\norm{\widehat{P_n}}_p$ would have order of magnitude
\begin{equation}
e^{Kn} n^{\frac{3}{2p}-1}.
\label{Pnorder}
\end{equation}
For $p=2$ we have from Parseval's theorem that $\norm{\widehat{P_n}}_2=\norm{P_n}_2$,
and by Theorem \ref{Lpestimate}, $\norm{P_n}_2 \sim 2^{-3/4} \pi^{-1/4} BC^{1/2} e^{Kn} n^{-1/4}$,
which is consistent with $\norm{\widehat{P_n}}_p$ having order of magnitude \eqref{Pnorder}.
In
Figure \ref{ell1Pn1to500} we plot $\frac{\norm{\widehat{P_n}}_1}{e^{Kn} n^{1/2}}$ for $n=1,\ldots,500$.

Since $Q_n$ has nonnegative Fourier coefficients, $Q_n(0)=\norm{\widehat{Q_n}}_1$, and
so $\norm{\widehat{Q_n}}_1=2^n$.  If $n^\alpha$ of the Fourier coefficients of $Q_n$ have magnitude on
the order of $2^n n^{-3/2}$ (which from Theorem \ref{Qninfinity} is the order of magnitude of $\norm{\widehat{Q_n}}_\infty$), then the identity $\norm{\widehat{Q_n}}_1=2^n$ implies that
$\alpha=\frac{3}{2}$. Then $\norm{\widehat{Q_n}}_p$ would have order of magnitude
\begin{equation}
2^n n^{\frac{3}{2p}-\frac{3}{2}}.
\label{Qnorder}
\end{equation}
By Theorem \ref{Qnestimate}, we have $\norm{Q_n}_2 \sim \Big( \frac{3}{\pi} \Big)^{\frac{1}{4}} 2^n n^{-3/4}$, and so by Parseval's theorem, $\norm{\widehat{Q_n}}_2 \sim \Big( \frac{3}{\pi} \Big)^{\frac{1}{4}} 2^n n^{-3/4}$, which is consistent with $\norm{\widehat{Q_n}}_p$ having order
of magnitude \eqref{Qnorder}. In Figure \ref{ell3Qn1to400} we plot
$\frac{\norm{\widehat{Q_n}}_3}{2^n n^{-1}}$ for $n=1,\ldots,400$.

\bibliographystyle{amsplain}
\bibliography{norms}

\end{document}